\documentclass[12pt,leqno,fleqn]{amsart}  
\usepackage{amsmath,amstext,amsthm,amssymb,amsxtra}
\usepackage[top=1.5in, bottom=1.5in, left=1.25in, right=1.25in]	{geometry}
\usepackage[T1]{fontenc}
\usepackage{lmodern}

\usepackage{euler}   
\usepackage[inline, shortlabels]{enumitem}

\usepackage{tikz}

\usepackage{mathtools}
\mathtoolsset{showonlyrefs,showmanualtags}

\usepackage{hyperref} 
\hypersetup{
    colorlinks=true,       
    linkcolor=blue,          
    citecolor=magenta,        
    filecolor=magenta,      
    urlcolor=cyan           
}

\usepackage[backrefs]{amsrefs}

\theoremstyle{plain} 
\newtheorem{lemma}[equation]{Lemma} 
 
\newtheorem{theorem}[equation]{Theorem} 
\newtheorem{corollary}[equation]{Corollary} 

\newtheorem{priorResults}{Theorem}

\theoremstyle{definition}

\theoremstyle{remark}

\numberwithin{equation}{section}

\title[Sparse Bounds for  Oscillatory Singular Integrals] {Sparse Bounds for Maximally Truncated \\ Oscillatory Singular Integrals} 

\author{Ben Krause}
\address{
Department of Mathematics
The University of British Columbia \\
1984 Mathematics Road
Vancouver, B.C.
Canada V6T 1Z2}
\email{benkrause@math.ubc.ca}
\thanks{Research supported in part by  an NSF Postdoctoral Research Fellowship.}

\author{Michael T. Lacey}   

\address{ School of Mathematics, Georgia Institute of Technology, Atlanta GA 3034, USA}
\email {lacey@math.gatech.edu}
\thanks{Research supported in part by grant  NSF-DMS-1600693.}

\begin{document}

\begin{abstract}
For polynomial $ P (x,y)$, and any Calder\'{o}n-Zygmund kernel, $K$, the operator below  satisfies a $ (1,r)$ sparse bound, for $ 1< r \leq 2$.  
\begin{equation}
\sup _{\epsilon >0} 
\Bigl\lvert 
\int_{|y| > \epsilon}  f (x-y) e ^{2 \pi i   P (x,y) } K(y) \; dy 
\Bigr\rvert
\end{equation}
The implied bound depends upon $ P (x,y)$ only through the degree of $ P$. 
We derive from this a range of weighted inequalities, including weak type inequalities on  $ L ^{1} (w)$, which are new, even in the unweighted case.  
The unweighted weak-type estimate, without maximal truncations,   is   due to Chanillo and Christ (1987).  
\end{abstract} 

	\maketitle

\section{Introduction}  

The Ricci  Stein \cites{MR822187,MR890662} theory of oscillatory singular integrals concern operators of the form 
\begin{equation*}
T_P f (x) = \int e( P  (x,y))K (y) f (x-y) \; dy , \qquad  e(t) := e^{2\pi i t}
\end{equation*}
where $ K (y)$ is a Calder\'on-Zygmund kernel on $ \mathbb R ^{n}$, and $ P \;:\; \mathbb R^n \times \mathbb R ^{n} \to \mathbb R$ is a polynomial of two variables.  
These operators are bounded on all $ L ^{p}$, with bounds that only depend upon the degree of the polynomial, the dimension, and the kernel $K$, 
an important point in the  motivations for this theory.

\begin{priorResults}\label{t:RS}  Under the assumptions above, there holds for a finite constant 
\begin{equation*}
\sup _{\textup{deg}(P)=d}
\lVert T_P  \;:\; L ^{p} \mapsto L ^{p}\rVert  <   \infty , \qquad 1<   p < \infty , 
\end{equation*}
\end{priorResults}

The $ L ^{1}$ theory is more delicate, with the dominant result being that of Chanillo and Christ  \cite{MR883667} proving that the operators $ T_P$ indeed map $ L ^1$ into weak $ L ^{1}$, with again the bound depending only on the degree of $ P$, the dimension, and the kernel $K$. Their argument does not address  maximal truncations. 

Our main result  proves sparse bounds for the maximal truncations 
\begin{equation}\label{e:monomial}
T _{P, \ast } f (x) = \sup _{\epsilon >0}
\Bigl\lvert 
\int _{\lvert  y\rvert > \epsilon  } f (x-y) e(P(x,y)) K(y) \; dy
\Bigr\rvert.  
\end{equation}
The bound that we prove implies the weak $L^1$ bounds for the maximal truncations, as well as quantitative bounds in $ A_p$, for $ 1\leq p < \infty $. 

 Call a collection of cubes $ \mathcal S$  in $ \mathbb R ^{n}$ \emph{sparse} if there 
are sets $ \{ E_S  \,:\, S\in \mathcal S\}$  
which are pairwise disjoint,   $E_S\subset  S$ and satisfy $ \lvert  E_S\rvert > \tfrac 14 \lvert  S\rvert  $ for all $ S\in \mathcal S$.
For any cube $ I$ and $ 1\leq r < \infty $, set $ \langle f \rangle_ {I,r} ^{r} = \lvert  I\rvert ^{-1} \int _{I} \lvert  f\rvert ^{r}\; dx  $.  Then the $ (r,s)$-sparse form $ \Lambda _{\mathcal S, r,s} = \Lambda _{r,s} $, indexed by the sparse collection $ \mathcal S$ is 
\begin{equation*}
\Lambda _{S, r, s} (f,g) = \sum_{I\in \mathcal S} \lvert  I\rvert \langle f  \rangle _{I,r} \langle g \rangle _{I,s}.  
\end{equation*}
Given a  sublinear operator $ T$, and $ 1\leq r, s  < \infty $, we set 
$ \lVert T \,:\, (r,s)\rVert$ to be the infimum over constants $ C$ so that for all  all bounded compactly supported functions $ f, g$, 
\begin{equation}\label{e:SF}
\lvert  \langle T f, g \rangle \rvert \leq C \sup  \Lambda _{r,s} (f,g), 
\end{equation}
where the supremum is over all sparse forms.  
It is essential that the sparse form be allowed to depend upon $ f $ and $ g$. But the point is that the sparse form itself varies over a class of operators with very nice properties.   

For singular integrals without oscillatory terms we have 

\begin{priorResults}\label{t:CZ} \cites{MR3521084,150105818}  Let $ K $ be a Calder\'on-Zygmund kernel on $ \mathbb R ^{n}$ as above. Then, the  operator  $ T f = \textup{p.v.} K \ast f (x)$  
satisfies $ \lVert T \,:\, (1,1)\rVert < \infty $.  
\end{priorResults}

Below, we obtain a quantitative version of a conjecture from  \cite{160906364}.

\begin{theorem}\label{t:Tsparse} For all integers $ d$,  and $ 1< r < 2$,  there holds 
\begin{equation}\label{e:Tsparse}
\sup _{\textup{deg} (P) \leq d} \lVert T _{P, \ast } \;:\; (1,r)\rVert \lesssim \tfrac 1 {r-1} . 
\end{equation}
The implied constant depends upon degree $ d$, dimension $ n$, and the kernel $ K$, but is otherwise absolute.  
\end{theorem}

As a corollary, we have these  quantitative weighted inequalities.  
The inequalities \eqref{e:wtd1} are new even for Lebesgue measure.  
The case of no truncations has been addressed in \cites{MR1782909,MR2900003}, but without effective bounds in terms of the $ A_p $ characteristic. 

\begin{corollary}\label{c:wtd} For every  $ d\geq 2$ and weight $ w\in A_1$ there holds 
\begin{gather}   \label{e:wtd1}
\sup _{\textup{deg} (P) \leq d} \lVert T _{P, \ast }   \,:\, L ^{1} (w) \mapsto L ^{1, \infty } (w)\rVert \leq  [w] _{A_1} ^{2} \log_+ [w] _{A_1}, 
\\  \label{e:wtd2}
\sup _{\textup{deg} (P) \leq d} \lVert  T _{P, \ast }   \,:\, L ^{p} (w) \mapsto L ^{p}  (w)\rVert \lesssim [w] _{A_p}  ^{1 + \max \{\frac 1 {p-1}, 1\}}, \qquad 1< p < \infty . 
\end{gather}
\end{corollary}

 Ricci-Stein theory has several interesting extensions.  
On the one hand, there are several  variants on the main result of \cite{MR1879821}, which considers an estimate which is uniform over the polynomials $ P$.  See \cites{MR2545246,11054504}.   One also has weighted extensions of the inequalities  
in for instance \cite{MR3291794}. 

Sparse bounds have recently been quite active research topic, impacting a range of operators. We point to the previously cited \cites{MR3521084,150105818}.  But also point to the range of operators addressed in \cites{160908701,2016arXiv161208881C,160506401,161103808,161203028,160305317}

\bigskip 
Our quantitative sparse bound in Theorem~\ref{t:Tsparse} closely matches the bounds obtained for `rough' singular integrals 
by Conde, Culiuc, Di Plinio and Ou \cite{2016arXiv161209201C}.  Their argument is a beautiful  abstraction of the methods of Christ  and others \cites{MR943929,MR951506,MR1317232}.  A large part of our argument can be seen as  an extension of \cite{2016arXiv161209201C}.

But, the oscillatory nature of the kernels present substantial  difficulties, and  additional new arguments are required to address maximal truncations.  Our prior paper \cite{160901564} proved the Theorem above in the special case of $ P (y) = y ^{d}$ in one dimension, and the interested reader will find that argument has fewer complications than this one.  
\begin{enumerate*}

\item  The essential oscillatory nature of the question is captured in Lemma~\ref{l:same0}. 
It has two estimates, the first \eqref{e:same} being well-known, having its origins in the work of Ricci and Stein \cite{MR822187}.
The second, \eqref{e:Z} is the additional feature needed to understand the $ L ^{1}$ endpoint. It is proved with the aid of arguments that can be found in the work of Christ and Chanillo \cite{MR883667}.  

\item The essential fact needed is the partial sparse bound of Lemma~\ref{l:2ndMain}.  This is a kind of `$L^1$ improving' estimate.    The initial steps in the proof of this Lemma depend upon a standard Calder\'on-Zygmund decomposition, with additional tweaks of the argument to adapt to the oscillatory estimates.

\item The further additional fact is Lemma~\ref{l:2}.  Crucially, a  the Carleson measure estimate 
is proved, which allows one to control the number of scales that impact this Lemma.  The maximal truncations are then controlled by  orthogonality considerations, and a general form of the Rademacher-Menshov theorem. 
\end{enumerate*}
 
We thank the referee for a careful reading.

\subsection{Notation}
As mentioned previously, here and throughout we use $e(t) := e^{2\pi i t}$; $M_{\textup{HL}}$ denotes the Hardy-Littlewood maximal function. For cubes $I \subset \mathbb{R}^n$, we let
$ \ell(I) := |I|^{1/n} $ denote its side-length.

With $d \geq 1$, we fix throughout the constant  
\begin{equation}\label{e:epsilon}
\epsilon_d := \frac{1}{2d}. 
\end{equation}

We use multi-index notation, $\alpha = (\alpha_1,\dots,\alpha_n)$, so
\[ x^\alpha := \prod_{i=1}^n x_i^{\alpha_i}.\]
 We use $\alpha > \beta$ to mean that $\alpha_i \geq \beta_i$ for each $1 \leq i \leq n$ with at least one inequality being strict.

\smallskip 
Recall that a Calder\'{o}n-Zygmund kernel $K$ on $ \mathbb R $, satisfies the following properties:
\begin{itemize}
\item $K$ is a tempered distribution which agrees with a $C^1$ function $K(x)$ for $x \neq 0$;
\item $\hat{K}$, the Fourier transform of $K$, is an $L^\infty$ function;
\item $|\partial^\alpha K(x)| \lesssim |x|^{-n - |\alpha|}$ for each multi-index $0 \leq |\alpha| \leq 1$. (We recall multi-index notation in the subsection on notation below.)
\end{itemize}
The key property of such kernels $K$ that we shall use is that we may decompose
\begin{equation}\label{decomp}
 K(x) = \sum_{j = -\infty}^\infty 2^{-nj} \psi_j(2^{-j}x), \ x \neq 0
\end{equation}
where $\psi_j$ are each $C^1$ functions supported in $\{ \frac{1}{4} < |x| \leq 1\}$, which satisfy 
\begin{equation}\label{psi} |\partial^\alpha \psi_j| \leq C \ \text{ for each multi-index } 0 \leq |\alpha| \leq 1 \text{ uniformly in $j$}
\end{equation}
and have zero mean,
$ \int \psi_j(x) = 0$. 
This decomposition is presented in \cite{MR1232192}*{Chap. 13}.  
\smallskip 

We will make use of the modified Vinogradov notation. We use $X \lesssim Y$, or $Y \gtrsim X$ to denote the estimate $X \leq CY$ for an absolute constant $C$. If we need $C$ to depend on a parameter,
we shall indicate this by subscripts, thus for instance $X \lesssim_p Y$ denotes
the estimate $X \leq C_p Y$ for some $C_p$ depending on $p$. We use $X \approx Y$ as
shorthand for $Y \lesssim X \lesssim Y $.

\section{Lemmas} \label{s:lemmas}

There are two categories of facts collected here,  (a) those which reflect the oscillatory nature of the problem, (b) a variant of the Rademacher-Menshov theorem. 

\subsection{Oscillatory Estimates}
This is a variant of the van der Corput lemma. 

\begin{lemma}\label{SW}\cite{MR1879821}*{Prop. 2.1.} 
Suppose $\Omega \subset \{ |x| \leq 1\}$ is a convex set, and 
$P(t) := \sum_{|\alpha| \leq d} \lambda_\alpha t^\alpha $ 
is a real polynomial, equipped with the coefficient norm, $\| P \| := \sum_{1 \leq |\alpha| \leq d} |\lambda_\alpha|$. Then for any $C^1$ function $\phi$, 
\begin{equation}\label{e:VDC}
\left| \int_{\Omega} e(P(t)) \phi(t) \ dt \right| \lesssim \| P\|^{-1/d}  \left( \sup_{|x| \leq 1} |\phi(x)| + \sup_{|x| \leq 1} |\nabla \phi(x)| \right).
\end{equation}
\end{lemma}

Next, a sublevel set estimate.  
\begin{lemma}\label{l:level} \cite{MR1879821}*{Prop. 2.2} We have the estimate below. 
\begin{equation*}
\lvert   \{  |x| \leq 1 \;:\;  \lvert  P  (x)\rvert < \epsilon  \}\rvert \lesssim \epsilon ^{1/d} \lVert P\rVert ^{-1/d} . 
\end{equation*}

\end{lemma}

Using a simple change of variables  we have for any cube   $I$, any convex set $\Omega \subset I$, and any $C^1$ function $ \varphi $, 
\begin{gather}\label{e:vdc}
\Bigl\lvert  \int _{\Omega } e (P  (x)) \varphi (x) \; dx  \Bigr\rvert 
\lesssim   \lvert  I\rvert ( \sup_{x \in I} |\varphi(x) | +  \ell(I) \sup_{x \in I} | \nabla \varphi(x) | )
\Bigl[ \sum_{\alpha } \ell(I) ^{\lvert  \alpha \rvert } \lvert  \lambda _{\alpha }\rvert  \Bigr] ^{-1/d} ,  
\\ \label{e:level} 
\lvert   \{  x \in I \;:\;  \lvert  P  (x)\rvert < \epsilon  \}\rvert \lesssim  \lvert  I\rvert \epsilon ^{1/d}
\Bigl[ \sum_{\alpha } \ell(I) ^{\lvert  \alpha \rvert } \lvert  \lambda _{\alpha }\rvert  \Bigr] ^{-1/d} .  
\end{gather}
Here, we let $\ell(I)$ denote the side length of the cube $ I$.

We will be concerned with operators that  have kernels 
\begin{equation}\label{e:phi}
 \phi _k (x,y) = 2^{-nk} e ( P (x,y) ) \psi _k (2^{-k} y), \qquad k \in \mathbb Z . 
\end{equation}
Above, the $ \psi _k  $ are as in \eqref{decomp}.  
This next lemma is the essential oscillatory fact, concerning a $ T ^{\ast} T$ estimate for convolution with respect to $ \phi _k$. 
Note that it holds for polynomials with no constant or linear term.

\begin{lemma}\label{l:same0}  Assume that (a) $ d\geq 2$, (b) the polynomial $ P$ does not have constant or linear terms, 
and is not solely a function of $ x$, 
and (c) $ \lVert P\rVert \geq  1$, and $ k \geq t_n$, for a dimensional constant $ t_n$.  

For each cube $ K$ with $  \ell (K) = 2 ^{k}$, there is a set $ Z_K \subset K \times K $ so that these three conditions hold. 
\begin{enumerate}
\item  For all $ t_n\leq j \leq k$, we have 
\begin{equation}\label{e:same}
\begin{split}
\mathbf 1_{K \times K} (x,y)\Bigl\lvert  &
\int \phi _k (x,z) \overline{ \phi _j (z,y)}  \;dz 
\Bigr\rvert
\\&\leq C_0 \{ 2 ^{-nk} \mathbf 1_{ Z_K } (x,y)+      2 ^{- (n+ \epsilon_d )k}  \mathbf 1_{K \times K} (x,y)\}. 
\end{split}
\end{equation}
Above, $ \tilde \phi (x,y) = \overline  \phi (-y)$, and $ \epsilon _d $ is as in \eqref{e:epsilon}. 
\item The sets $Z_K$ have the following ``small-neighborhoods'' property: for any $1 \leq 2^s \leq 2^k$,
\begin{equation}\label{e:Z0}
|Z_K + \{|(x,y)| \leq 2^s \}| \lesssim |K|^2(2^{-\epsilon_dk} + 2^{s-k});
\end{equation}
here we are taking the Minkowski sum of $Z_K$ and the $2^s$ ball;
\item  In fact, the sets $ Z_K  $  satisfy the fiber-wise estimate 
\begin{equation}\label{e:Z} 
\sup_{x\in K}
| \pi_x Z_K + \{ (0,y)  \;:\;   \lvert  y\rvert  \leq 2^s \}| \lesssim  2 ^{nk} ( 2^{- \epsilon_d k} + 2^{s-k} ),  
\qquad 1 \leq 2^s \leq 2 ^{k},
\end{equation}
where $\pi_x Z = \{ y \;:\; (x,y)\in Z\}$, is the $x$-fiber of $Z$, and  
in \eqref{e:Z} we are taking the Minkowski sum of $ Z_K$ and a ball in $ \mathbb R ^{n}  $, and measure is taken in $\mathbb R^n$.   
\end{enumerate}

\end{lemma}

The estimate \eqref{e:same} is uniform in $ 1\leq j \leq k$, and the  right side has two terms. 
The second term is the one in which we have additional decay in the convolution, beyond what we naively expect. 
The first term involving $ Z_K$ is that term for which we do not claim any additional decay in the convolution. 
Thus, additional information about the set $ Z_K$ is needed, which is the content of \eqref{e:Z}. 
Part of this information is well-known: $ Z_K$ has small measure \eqref{e:Z0}.  
The more refined information in \eqref{e:Z} is that \emph{on each fiber},  the measure of a neighborhood of the set 
is small. 
This condition is not formulated  by Chanillo and Christ \cite{MR883667}, but follows from their techniques, which we present below.

\begin{proof}
Write the polynomial $ P (x,y)$ as 
\begin{equation} \label{e:P}
P (x,y) = \sum_{\alpha, \beta \;:\; \lvert  \alpha \rvert + \lvert  \beta \rvert \geq 2,\ \beta \neq 0  } \lambda _{\alpha, \beta  } x ^{\alpha } y ^{\beta }, 
\end{equation}
where $ \lVert  \lambda \rVert = \sum_{\alpha , \beta } \lvert  \lambda _{\alpha , \beta }\rvert =1 $. 
We will use the van der Corput estimate \eqref{e:vdc} to estimate the integral in $ z$ in \eqref{e:same}.
The integral in \eqref{e:same} is explicitly 
\begin{equation*}
  2^{-jn - kn} \int e (  P (x,z) - P (y,z) ) \psi _j (2^{-j}  (z-y)) \psi_k (2^{-k} (x-z)  ) \; dy 
\end{equation*}
Write  $ P (x,y) = \sum_{\alpha, \beta  } \lambda _{\alpha, \beta  } x ^{\alpha } y ^{\beta }$ the phase function above as 
\begin{gather*}
 P (x,z) - P (y,z)= \sum_{ \beta  \;:\; \lvert  \beta \rvert >0 }  [R _{\beta  }  (x)- R_{\beta} (y)] z ^{\beta  }, 
\\
 \textup{where} \quad R _{\beta  } (x)  = \sum_{\alpha \;:\; \lvert  \alpha \rvert + \lvert  \beta \rvert >1    } \lambda _{\alpha , \beta } x ^{\alpha  }. 
\end{gather*}
Above we have $ \lvert  \alpha \rvert \geq 1  $.   Then, by \eqref{e:vdc},  
\begin{equation*}
\textup{LHS of \eqref{e:same}} \lesssim  \lvert  I\rvert ^{-1}   \Bigl[  \sum_{\beta  \;:\;  \lvert  \beta  \rvert \geq 1  }   (\ell I) ^{\lvert  \beta  \rvert  } \lvert R _{\beta  }  (x) -R_{\beta} (y)\rvert \Bigr]  ^{- 1/ d }.    
\end{equation*}
Therefore, we  take the set $ Z_I$ to be 
\begin{equation} 
\label{e:Zk}
Z_I =  \Bigl\{   (x,y)  \in I \times I\;:\;  \sum_{\beta }   (\ell I) ^{\lvert  \beta \rvert  } \lvert R _{\beta  }  (x) -R_{\beta} (y)\rvert < 2 ^{k/2} \Bigr\}.  
\end{equation}
We see that   that \eqref{e:same} holds. 
This completes the first part of the conclusion of the Lemma. 

For the second part, the set $ Z_I$ in \eqref{e:Zk} is contained in the set $ \bigcup _{\sigma } Z (\sigma )$, where 
\begin{equation*}
Z (\sigma ) :=  \Bigl\{  (x,y) \in I \times I  \;:\; 
\Bigl\lvert \sum_{\beta } \sigma (\beta )  \ell (I) ^{\lvert  \beta \rvert }
[R _{\beta  }  (x)- R_{\beta} (y)]
 \Bigr\rvert <  2 ^{k/2}
\Bigr\}. 
\end{equation*}
The union is over all choices of signs $ \sigma \;:\; \{ \beta \} \mapsto  \{\pm 1\}$. There are $ O (2 ^{nd}) = O (1)$ such choices of $ \sigma $. 
Fixing $ \sigma $, the polynomial of two variables
$\sum_{\beta } \sigma (\beta )  \ell (I) ^{\lvert  \beta \rvert }
[R _{\beta  }  (x)- R_{\beta} (y)]$ has norm at least one, since $\| P \| \geq 1$; similarly, if we fix $x$ as well, the polynomial of one variable $ \sum_{\beta } \sigma (\beta )  \ell (I) ^{\lvert  \beta \rvert } R _{\beta }(y)$ 
has norm at least one.
It follows from Lemma~\ref{l:neighborhood} below, applied in dimensions $2n$ and $ n$, that the set $ Z (\sigma )$ satisfies the estimates \eqref{e:Z0} and \eqref{e:Z}. 
That  completes the proof. 
\end{proof}

This is the main point that remains to be addressed. 

\begin{lemma}\label{l:neighborhood}   Let $ P = P (x)$ be a polynomial on $ \mathbb R ^{n}$  with $ \lVert P \rVert  \geq 1$ and degree $d$. 
Let   $ I$ be a cube of side length $ \ell (I) = 2 ^{k} \geq  2 ^{t_n}$,  and $ 2 ^{t_n} \leq 2 ^{s} \leq 2 ^{k}$. We have the estimate 
\begin{equation}\label{e:Neighborhood}
\lvert  Z_I + \{ x  \;:\;   \lvert  x \rvert  \leq 2^s \}\rvert \lesssim  
\lvert  I\rvert  
\bigl\{
 2 ^{-\epsilon _d k/2}   +    2 ^{s - k}
\bigr\}. 
\end{equation}
where $ Z_I =  \{  x \in I \;:\; \lvert  P (x)\rvert < \ell (I) ^{1/2}  \}$.   
\end{lemma} 

The case of dimension $n=1 $ is easy.  The set $ Z_I$ has small measure by the van der Corput estimate \eqref{e:vdc}. 
But, it is the pre-image of an interval under a degree $ d$ polynomial $ P$.  Hence it has $ O (d) = O (1)$ components. 
From this, \eqref{e:Neighborhood} is immediate.  

\smallskip

The higher dimensional case   requires some additional insights, because level sets in two and higher dimensions are, in general, unbounded algebraic varieties.   
We need the following Lemma, drawn from Chanillo and Christ  \cite{MR883667}. By a $k$-strip we mean the set 
\begin{equation}\label{e:strip}
 S = \bigcup_{j \in \mathbb{Z}} Q + 2^k (0, \dots, 0, j) , \qquad  \textup{$ Q$ is a cube}
\end{equation}
 By a $k$-interval we mean a (possibly infinite) subset of $S$ given by
\[ I = \bigcup_{j_0 <j < j_1} Q+ 2^k (0, \dots, 0, j), \ \text{ for } j_0, j_1 \in \{\pm \infty \} \cup \mathbb{Z}.\]

\begin{lemma}[\cite{MR883667}, Lemma 4.2] \label{l:strips}
For any dimension $n$ and degree $d$, there is a $C \lesssim_{d,n} 1$ so that for any $A > 0$, and any polynomial $P$ of degree $d$, and any $k$-strip $S$, the subset of $S$ given by
\[ \bigcup \{ Q \in \mathcal{D}_k^{\vec{\omega}}: Q \subset S, Q \cap \{ |P(x)| < A \} \neq \emptyset \} \]
is a union of at most $C$ $k$-intervals.
\end{lemma}

The Lemma above is proven for $k=0$,   but as the result in \cite{MR883667} holds for polynomials of arbitrary norm, the statement therein implies the one above.  It likewise holds for any rotation of a strip, which we will reference shortly.  

We will also recall the following behavior of the coefficient norm $\| P\|$ under the action of the orthogonal group, $\mathcal{O}(n)$.

\begin{lemma}\label{l:tech} For any degree $d$, and any $n \geq 2$, if $P$ has degree $d$, then
\[ \| P \| \approx_{d,n} \| P \circ \theta \| \]
for any $\theta \in \mathcal{O}(n)$.
Moreover,   for $ d\geq 2$, there exists some $\theta = \theta(P) \in \mathcal{O}(n)$ so that
 for any choice of $  1 \leq j < k  \leq n$, there holds 
\begin{equation}\label{e:lower}
\| \partial_j ( P \circ \theta) \| 
\gtrsim_{d,n} \| P \| .
\end{equation}

\end{lemma}

\begin{proof}
We argue by compactness and contradiction.  If the conclusion does not hold for some choice of dimension $ n$ and degree $ d$, for all integers $ j$, we can select $ P _{\lambda _j}$ and $ \theta _j$ so that $ \lVert P _{\lambda _j}\rVert=1$ and $ \lVert P _{\lambda _j}  \circ \theta _j\rVert < 1/j$.  For some subsequence, we must have $ \{ \lambda _{\alpha ,j} \;:\; \lvert  \alpha \rvert \leq d \} \to \{\lambda _ \alpha \;:\; \lvert  \alpha \rvert \leq d  \}$, and $ \theta _j \to \theta $.  We conclude that $ \lVert P _{\lambda }\rVert = 1$ and $ \lVert P _{\lambda } \circ \theta \rVert=0$, which is a contradiction. 

\bigskip

Turning to the second claim, let $ \mathcal D$ be the collection of differential operators 
\begin{equation*}
\partial _j \quad 1\leq j \leq n. 
\end{equation*}
We argue  again, by contradiction and compactness.  
For some choice of dimension $ n$ and degree $ d$,  there is a polynomial $ P$ with $ \lVert P\rVert=1$ so that 
for all $ \theta \in \mathcal O (n)$, there is a choice of $  D \in \mathcal D$ so that  
$   D(P \circ \theta)    =0$. 
The map $ \theta \to  D (P \circ \theta ) $ is continuous, so that the set $\Theta _D=  \{\theta \;:\;   D (P \circ \theta)    =0\}$ is closed. 

We also have $ \bigcup _{D\in \mathcal D} \Theta _D  = \mathcal O (n)$. The Baire Category Theorem implies that  for some $ D$, there is a  $ \theta _0$ in the interior of $ \Theta  _{D}$. Hence,  for all $ \theta $ sufficiently close to  $ \theta _0$, we have 
$  D P \circ \theta  =0$.  But the degree of $ P$ is at least 2, so this is can only happen if $ P$ itself is zero, which is a contradiction.

\end{proof}

\begin{proof}[Proof of Lemma \ref{l:neighborhood}]
For $s \in \mathbb{Z}$, and subsets $A \subset \mathbb{R}^n$, set  $ A^s := A + \{ |x| \leq 2^s \}$.

We  prove: For each $P$ of degree at most $d \geq 1$, $\| P \| \gtrsim_{d,n} 1$, and any $k \geq 0$,  
and any cube $ I$ with $ \ell (I) = 2 ^{k}$,  for the set $ Z=\{  x \in I : |P(x)| \leq 2^{k/2} \}$, we have 
\begin{equation}\label{goal}
 | Z ^s| \lesssim_{d,n} 2^{nk}( 2^{-\epsilon_{d}k} + 2^{s-k} ), \qquad  1 \leq s \leq k. 
\end{equation}
In view of our definition of $ Z_I$ in \eqref{e:Zk},  and the condition $ \lVert Q _{\beta _0}\rVert \gtrsim 1$, this proves \eqref{e:Neighborhood}.  

\smallskip 

We will induct on the degree of the polynomial, so let us first assume that $d=1$, and that $P$ is linear. But then, the 
set $ Z  $ is a of the form $ \{  x \in I  \;:\; \lvert   \langle \xi , x  \rangle\rvert < 2 ^{k/2} \}$, for some choice of $ \lVert \xi \rVert \approx 1$.  It is clear that \eqref{goal} holds.  

Henceforth, we will assume that $d \geq 2$. Now, since \eqref{goal} is invariant under replacement of $x$ by $\theta x$, for $\theta \in \mathcal{O}(n)$, we may assume by Lemma \ref{l:tech} that 
the condition \eqref{e:lower} holds.  
Thus, the induction hypothesis applies to each polynomial $ \partial _{r} P$,  $ 1\leq r \leq n$.
As a consequence, we have this.  
\begin{gather}\label{e:E}
|E^{s+c_n} | \lesssim_{d,n} 2^{nk}( 2^{-\epsilon_{d-1}k} + 2^{s-k} ) \leq 2^{nk} ( 2^{- \epsilon_{d} k} + 2^{s-k} ), \qquad  1 \leq s \leq k, 
\\
\textup{where} \quad  E := \bigcup_{r=1}^n \{ |x| \leq  2^{k+1} : | \partial_r P(x) | \leq 2^{k/2} \}. 
\end{gather}
Above, $ c_n$ is a dimensional constant.  
Observe that each of the sets that we form a union over, when restricted to a strip, can be   covered by $ C _{n,d}$  intervals, by Lemma~\ref{l:strips}. Thus, the same conclusion holds for the union.

The set we need to estimate is the set $ Z  ^{s}$, but off of the set $ E ^{s+c_n}$, which we denote by 
$ X  :=  Z  ^{s} \setminus E ^{s+c_n}$. For each $\sigma: \{ 2,\dots, n\} \to \{ \pm 1\}$, let $\theta_\sigma \in \mathcal{O}_n$ be such that
\[ \theta_\sigma^{-1} \vec{e_n} = \frac{1}{\sqrt{n}} (\sigma(1),\dots, \sigma(n)). \]
Here, $\vec{e_n} := (0,\dots, 0,1)$ is the $n$th basis vector.

With this in mind, write
\[ X = \bigcup_\sigma Y_\sigma,\]
where 
\[ \aligned 
Y_\sigma &:= \{ |x| \lesssim 2^k : d(x,E) > 2^{s+c_n}, |P(x)| \leq 2^{k/2}, &
\\ & \qquad \qquad  \sigma(r) \partial_r P(x) > 2^{k/2} \text{ for all } 1\leq r \leq n \} \\
&=
\{ |x| \lesssim 2^k : d(x,\theta_\sigma^{-1} E) > 2^{s+c_n}, |P(\theta_\sigma x)| \leq 2^{k/2}, 
\\ & \qquad \qquad 
\sigma(r) (\partial_r P)(\theta_\sigma x) > 2^{k/2} \text{ for all } 1\leq r \leq n \}. \endaligned\]
We will now favorably estimate $|Y_\sigma| \lesssim 2^{nk}(2^{-\epsilon_d k} + 2^{s-k})$ for each choice of $\sigma$.

Since any cube of dyadic side length is contained in a dyadic cube of six times its length shifted by some element $\vec{\omega} \in \{ 0, 1/3\}^n$, it suffices to show that for any grid shifted by any $\vec{\omega} \in \{ 0, 1/3\}^n$, and any $s+3$ strip, $S = S_{\vec{\omega}}$, situated in that grid, $S$ meets $Y_{\sigma}$ in at most a bounded number of cubes. Since finite unions and complements of $s+3$ intervals are expressible as finite unions of $s+3$ intervals, it suffices to prove this result for $s+3$ intervals. In particular, it suffices to show that any interval $I \subset S$ that meets $Y_{\sigma}$ does so in at most $2$ cubes.

So, suppose now that $Q \in I$, and $Q \cap Y_\sigma \neq \emptyset$. Since we have excised a $2^{s + c_n}$ neighborhood of $E$, this implies that for every $x \in Q$, $x \notin \theta_{\sigma}^{-1} E $, and thus
\[ |\partial_r P(\theta_\sigma x)| > 2^{k/2} \text{ for all } 1 \leq r \leq n. \]
But, we know that there exists some point $y_Q \in Q \cap Y_\sigma$; for this $y_Q$, we have
\[ \sigma(r) (\partial_r P)(\theta_\sigma y_Q) > 2^{k/2} \text{ for all } 1 \leq r \leq n,\]
so by connectedness and continuity, it follows that
\[ \sigma(r) (\partial_r P)(\theta_\sigma x) > 2^{k/2} \text{ for all } 1 \leq r \leq n \]
for each $x \in Q$.
But now we see that
\[ \aligned 
\partial_n ( P(\theta_\sigma x)) &= \vec{e_n} \cdot \theta_\sigma( \nabla P)(\theta_\sigma x) = 
\theta_\sigma^{-1}(\vec{e_n}) \cdot (\nabla P)(\theta_\sigma x) \\
&= \frac{1}{\sqrt{n}} \sum_{r=1}^n \sigma(r) (\partial_r P)(\theta_\sigma x) > \sqrt{n} 2^{k/2}. \endaligned\]
This strong monotonicity yields the result.

\end{proof}

\subsection{Rademacher-Menshov Theorem}
There is a general principle, a variant of the Rademacher-Menshov inequality that we will reference to control  maximal 
truncations.   This has been observed many times, for an explicit formulation and proof, see  \cite{MR2403711}*{Thm 10.6}. 

\begin{lemma}\label{l:RM}  Let $ (X, \mu )$ be a measure space, and $ \{ \phi _j \;:\; 1\leq j \leq N\}$ a sequence of functions 
which satisfy the Bessel type inequality below, for all sequences of coefficients $c_j \in \{ 0, \pm 1\}$, 
\begin{equation}\label{e:bessel}
\Bigl\lVert \sum_{j=1} ^{N}  c_j \phi _j \Bigr\rVert _{L ^2 (X)} \leq A .  
\end{equation}
Then, there holds 
\begin{equation}\label{e:RM}
\Bigl\lVert\sup _{1< n \leq N} 
\Bigl\lvert 
 \sum_{j=1} ^{n}   \phi _j
\Bigr\rvert
\Bigr\rVert _{L ^2 (X)} \lesssim A   \log(2+ N) .  
\end{equation}

\end{lemma}

\section{The Main Lemma} \label{s:proof}
The polynomials $ P (x,y)$ in Theorem~\ref{t:Tsparse} are general polynomials.  But, we can 
 without loss of generality assume that $ P$ does  not contain (a) constants,  (b) terms that are purely powers of $ x$, 
 nor (c) terms that are linear in $ y$.  
 That is, we can write 
\begin{equation*}
 P  (x,y) = \sum_{\substack{\alpha , \beta  \;:\; 2\leq \lvert  \alpha \rvert + \lvert  \beta \rvert \leq d\\    \lvert  \beta\rvert  \neq 0} } \lambda _{\alpha, \beta  } x ^{\alpha } y ^{\beta }. 
\end{equation*}
Define  $ \lVert P\rVert= \sum_{\alpha , \beta } \lvert  \lambda _{\alpha , \beta }\rvert $.
Observe: Any dilate of a Calder\'on-Zygmund kernel is again a Calder\'on-Zygmund kernel.  Therefore, in proving our sparse bounds, it suffices to do so for a polynomials satisfying $ \lVert P\rVert=1$.  We will do so using induction on degree. 
Notice that the induction hypothesis implies that the sparse bounds  hold without restriction on $ \lVert P\rVert$. These remarks are important to the proof.

The essential step is to show that the sparse bounds  of Theorem~\ref{t:Tsparse} holds in these cases.  

\begin{lemma}\label{l:cases} 

The operators $ T_{P,\ast } $ satisfy the sparse bounds \eqref{e:Tsparse} under either of  these assumptions.  

\begin{enumerate}
\item   The polynomial $ P (x,y) = P (y)$ is only a function of $ y$.  

\item  The polynomial $ P$ satisfies $ \lVert P\rVert = 1$, and the kernel $ K (y) $ of the operator $ T$ is supported 
on $ \lvert  y\rvert \geq  2 ^{t_n}  $, where $ t_n \lesssim 1$ is a dimensional constant.  
\end{enumerate}

\end{lemma}

We take up the proof of the Lemma, returning to the conclusion of the proof our main Theorem~\ref{t:Tsparse} at the end of this section.  Now, in the case of $ P$ being only a polynomial of $ y$, the conclusion is invariant under dilations, so that we are free to assume that in this case $ \lVert P\rVert   =1 $.  
We need only concern ourselves with `large scales.'  
For any finite constant $ t_n$, the  operator defined below, using the notation \eqref{decomp}, 
\begin{equation*}
f \mapsto  \int e (P(y)) f (x-y)   \sum_{ j <  t_n } 2^{-jn} \psi_j(2^{-j} y) \; dy
\end{equation*}
is a Calder\'on-Zygmund operator, hence its maximal truncations are bounded on $ L ^{1}$ to weak $ L ^{1}$, with 
a norm bound that only depends upon the polynomial $ P $ through its degree.  

\smallskip 

Having removed that part of the kernel close to the origin, both cases in Lemma~\ref{l:cases} fall under the assumptions of case 2. 
The  maximal truncations  are at most 
\begin{equation} \label{e:M1}
\sup _{k_0 \geq C_d} 
\Bigl\lvert 
\sum_{k=k_0} ^{\infty } \int e (P(x,y)) 2^{-kn} \psi _k (2^{-k} y) f (x-y) \; dy 
\Bigr\rvert +M_{\textup{HL}} f =: \tilde T _{\ast } f + M_{\textup{HL}} f.
\end{equation}
On the right, the maximal function  admits a sparse bound of type (1,1),
so we show the sparse bound $ (1,r)$  operator $ \tilde T _{\ast } f $.

\smallskip 
We make a familiar dyadic reduction, using \emph{shifted dyadic grids}. 
For each $\vec{\omega} \in \{0,1/3,2/3\}^n$, let $\mathcal{D}^{\vec{\omega}}$ be the cubes in 
\[ \{ 2^k( [0,1]^n + \vec{m} + (-1)^k{\vec{\omega}}) : \vec{m} \in \mathbb{Z}^n , k \in \mathbb Z \},\]
that is the dyadic grid  $ \mathcal D$ shifted by ${\vec{\omega}}$.
It is well known that for any cube, $I$, there exists some ${\vec{\omega}} = {\vec{\omega}}(I) \in \{0,1/3,2/3\}^n$ and  some $P = P(I) \in \mathcal{D}^{\vec{\omega}}$, we have the containment $ I \subset P$, and $ \ell (P) \leq 6 \ell (I)$.  
Moreover, fixing the side length of a cube, we can resolve the identity function by 
\begin{equation} \label{e:=1}
\sum_{\vec \omega \in \{0,1/3,2/3\}^n } 
\sum_{I\in\mathcal{D}^{\vec{\omega}} \;:\; \ell (I) = 2 ^{k}}
\mathbf 1_{\frac 13 I } \equiv 1, \qquad k\in \mathbb Z . 
\end{equation}
Observe that we then have 
\begin{gather*}
\int e ( P(x,y)) 2^{-kn} \psi _k (2^{-k} (x-y) ) f (y)\; dy 
= 
\sum_{\vec \omega \in \{0,1/3,2/3\}^n } \sum_{\substack{I \in  \mathcal D ^{\vec \omega }  \\ \ell(I) = 2^{k+ t_n}} }
T _{I} f (x),
\\
\textup{where} \qquad 
T _{I} g (x) = \int e ( P(x,y)) 2^{-kn} \psi _k (2^{-k} (x-y) )  (g   \mathbf 1_{ \frac 13 I } ) (y)\; dy. 
\end{gather*}
\begin{align} 
\label{e:T8}
T _{\ast , \vec \omega } f &= \sup _{\epsilon  \geq 2^{C_d} } \Bigl\lvert  \sum_{I \in  \mathcal D ^{\vec \omega } \;:\; \ell(I) \geq \epsilon  } T_I f  \Bigr\rvert
\end{align}
The role of the grid  $ \mathcal{D}^{\vec{\omega}}$  in the remaining argument is of a   standard nature, and so we suppress $ \vec \omega $ in the notation in the argument to follow.

\smallskip 
We will freely decompose the collection of cubes $ \mathcal D_+ =  \{I \in  \mathcal D ^{\vec \omega } \;:\; \ell(I) \geq 2 ^{C_d}\}$.  Thus, extend the notation \eqref{e:T8} to 
\begin{equation} \label{e:TI}
T _{\ast ,  \mathcal I} f = \sup _{\epsilon  >0 } \Bigl\lvert  \sum_{I \in  \mathcal I \;:\; \ell(I) \geq \epsilon  } T_I f  \Bigr\rvert
\end{equation}
where $ \mathcal I \subset \mathcal D_+$.  Our second main Lemma is as below. As it forms the core of the proof, we place it's proof in the next section.

\begin{lemma}\label{l:2ndMain} Suppose that $ f, g $ are supported on cube $ I_0$, and  and $ \mathcal I$  is a collection of subcubes of $ I_0$ for which 
\begin{equation}\label{e:K}
\sup _{I\in \mathcal I} \langle f \rangle_I < A \langle f \rangle_{I_0}, 
\qquad 
\sup _{I\in \mathcal I} \langle g \rangle_I < A \langle g \rangle_{I_0}, 
\end{equation}
where $ A \lesssim 1$ is a constant.  Then, 
\begin{equation}\label{e:2ndMain}
\langle T _{\ast , \mathcal I} f, g \rangle \lesssim \tfrac 1 {r-1}\lvert  I_0\rvert   \langle f \rangle _{I_0} \langle g \rangle _{I_0,r}, 
\qquad 1< r \leq 2. 
\end{equation}
The implied constant depends upon $A$, and $ P (x,y)$ only through the degree of $ P$.  
\end{lemma}

\begin{proof}[Proof of Lemma~\ref{l:cases}]
Recall that we are to prove the sparse bound for $ T _{\ast , \mathcal D_+}$, as defined in \eqref{e:TI}. 
It suffices to consider bounded functions $ f, g$ supported on a fixed cube $ I_0 \in \mathcal D_+$.  
Now, it is easy to see that 
\begin{equation*}
\sum_{J \;:\; J\supset I_0}  \lvert  T_J f\rvert \mathbf 1_{I_0}   \lesssim \langle f \rangle _{I_0} . 
\end{equation*}
It suffices to prove the sparse bound for $ T _{\ast , \mathcal I_0}$, where $ \mathcal I_0$ consists of the dyadic cubes strictly contained in $ I_0$.  

Add the cube $ I_0$ to the sparse collection $ \mathcal S$.  
Take $ \mathcal E$ to be the maximal dyadic cubes $ P\subset I_0$ so that at least one of the following two inequalities hold: 
\begin{equation*}
\langle f \rangle _{P} > 100 \langle f \rangle _{ I_0}
\quad \textup{or} \quad 
\langle g \rangle _{P} > 100 \langle g \rangle _{ I_0}. 
\end{equation*}
Let $ E= \bigcup \{P \;:\; P\in \mathcal E\}$, so that $ \lvert  E\rvert \leq \frac 1 {50} \lvert  I_0\rvert  $. 
And let $ \mathcal I = \{I\in \mathcal I_0 \;:\;   I\not\subset E\}$. It follows that 
\begin{align*}
\langle T _{\ast , \mathcal I_0} f, g \rangle 
\leq \langle T _{\ast , \mathcal I} f, g \rangle + 
\sum_{P}\langle T _{\ast , \mathcal I_0 (P)} f, g \rangle , 
\end{align*}
where $ \mathcal I_0 (P) = \{I\in \mathcal I_0 \;:\; I\subset P\}$.  

The first term on the right is controlled by \eqref{e:2ndMain}. And, we add the collection $ \mathcal E $ to the sparse collection $ \mathcal S$, and recurse on the second group of terms above. This completes the proof of the sparse bound. 

\end{proof}

 To complete the proof of the Theorem~\ref{t:Tsparse}, we need to consider the case \emph{not covered} by Lemma~\ref{l:cases}, namely 
 
\begin{lemma}\label{l:lastCase}  The operator $ T _{P, \ast }$  satisfies the sparse bound inequalities \eqref{e:Tsparse}  under the assumptions that polynomial $ P$ satisfies $ \lVert P\rVert = 1$ and    the kernel $ K (y)$ of the operator $ T$ is supported on $ \lvert  y\rvert \leq   2 ^{t_n}  $.   
\end{lemma}

\begin{proof}
We induct on the degree of the polynomial $ P (x,y)$ in the $ x$-coordinate, call it $ d_x$. 
The case of $ d_x=0$ is contained in the first case of Lemma~\ref{l:cases}, which we use as the base case. 

We pass to the inductive case of $ d_x >0$.  
Note that the induction hypothesis implies that we have the full strength of our main theorem for polynomials of degree $ d_x-1$ in the $ x$-coordinate.
Now, the the kernel $ K $ is supported on the  cube $ I_0 = [-2 ^{-t_n}, 2 ^{t_n}] ^{n} $, hence it suffices to prove the sparse bounds for functions $ f $ supported on a cube $ m+ I_0$, uniformly over $ m\in \mathbb Z ^{d}$.  
Equivalently, it is the same to prove the inequality functions supported on $ I_0$, uniformly over polynomials $ P (m+x,y) $, 
where $ P$ is a fixed polynomial of degree $ d_x$ in the $ x$-coordinate, and $ m\in  2 ^{t_n+1}\mathbb Z ^{n}$. 
Write 
\begin{align*}
R_m (x,y) = P (m+x,y) -  P (x, y). 
\end{align*}
This is a polynomial with degree in $ x$ at most $ d_x -1$. (In fact, for $ m=0$, it is the zero polynomial.)  Hence, $ T _{ R_m,\ast}$ satisfies the sparse bounds, uniformly in $ m\in \mathbb Z ^{n}$.

But, note that for $ x, y \in I_0$, 
\begin{equation*}
\lvert  e (P (m+x,y)) - e (R_m (x,y))\rvert = \lvert  e (P (x,y)) -1\rvert \lesssim \lvert  y\rvert,   
\end{equation*}
since $ \lVert P\rVert \leq 1$.   Therefore, we have 
\begin{equation*}
\lvert  T _{ P (m+ \cdot , \cdot ),\ast} - T _{ R_m,\ast} f (x)\rvert \lesssim M f .  
\end{equation*}
The maximal function also satisfies the $ (1,1)$ sparse bound, so the proof is complete. 
\end{proof}

\section{Proof of Lemma~\ref{l:2ndMain}} 
We begin with a Calder\'on-Zygmund decomposition.  
Let $ \mathcal B$ be the maximal subcubes  $J\subset I_0$ for which 
$
\langle f \rangle_J \geq  A  \langle f \rangle_{I_0} 
$.
Write $ f = \gamma + b$ where 
\begin{equation*}
b = \sum_{J\in \mathcal B} f   \mathbf 1_{J} = \sum_{s=0} ^{s_0}  \sum_{J\in \mathcal B_s} f\mathbf 1_{J} =: 
 \sum_{s=0} ^{s_0}  b_s. 
\end{equation*}
Above, we set $ \mathcal B_0 = \{ J\in \mathcal B \;:\; \ell (J) \leq 0 \}$, and for $ 0< s \leq s_0 = \log_2 \ell(I_0)-1$, set 
$ \mathcal B _s = \{ J\in \mathcal B \;:\; \ell (J) =s  \}$.  No cancellation property of $b$ is needed. 
And, as a matter of convenience, we set $ B_s \equiv 0$ if $ s <0$. 

The first step to to observe by Lemma~\ref{l:simple} below we have 
\begin{equation*}
\langle T _{\ast , \mathcal I} \gamma , g \rangle 
\lesssim \lVert T _{\ast , \mathcal I}  \;:\; L ^{r'} \to L ^{r'}\rVert \cdot \lVert \gamma  \rVert _{r'} \lVert g\rVert_r 
\lesssim \tfrac1 {r-1} \lvert  I_0\rvert \langle f \rangle _{I_0}  \langle g \rangle _{I_0, r}, 
\end{equation*}
 since $ \lVert \gamma \rVert _{\infty } \lesssim \langle f \rangle _{I_0}$.  Note in particular that we have written the 
 norms on $ f$ and $ g$ in sparse form.  

 It remains to consider the bad function.  Observe that if $ J\in \mathcal B$ and $ K\in \mathcal I$, 
 we have either $ J\cap K = \emptyset $ or $ J\subsetneq K $.  Therefore, we can write 
 \begin{align*}
T _{K} b &= \sum_{s  \;:\;   1\leq 2 ^{s} < \ell (K)} T _K b_s  
= \sum_{s =1} ^{ \infty } T_K b _{k-s}
\\
&=: \sum_{s =1} ^{ \infty } T_{K,s} b  , \qquad \ell (K) = 2 ^{k}.  
\end{align*}
We will consistently assume that $ \ell (K) = 2 ^{k}$ below.

For integers $ 0 \leq s \leq s_0$ we decompose $ \mathcal K  = \mathcal S_s \cup \mathcal N_s$ where $ K\in \mathcal S_s$ if 
$ 2 ^{s} < \ell (K)$ and there holds 
\begin{equation}\label{e:standard}
\lVert T_{K,s} b \rVert_2 ^2 < 100 C_0 \lvert  K\rvert ^{-(1+ \epsilon _d /n)} \lVert b_{k-s} \mathbf 1_{K}\rVert_1 ^2 ,  
\end{equation}
 where $ C_0$ is the  constant in \eqref{e:same}.    We refer to $ \mathcal S_s$ as the `standard' collection for which the second, simpler, term in \eqref{e:same} is decisive.  
 
\begin{lemma}\label{l:standard} 
We have these inequalities 
\begin{equation} 
\Bigl\lVert  \sup _{\epsilon } \Bigl\lvert  
\sum_{s \geq 0 \;:\; } \sum_{K \in \mathcal S_s \;:\; \ell (I) \geq \epsilon } T_{K,s} b 
\Bigr\rvert \Bigr\rVert_q \lesssim q \langle f \rangle_{I_0} \lvert  I_0\rvert ^{1/q}.  
\end{equation}
\end{lemma}
 
\begin{proof}
We have a gain in the scale. Holding the side length of $ K$ fixed, it is clear that for any integer $ j$
\begin{equation*}
\Bigl\lVert  
\sum_{s \geq 0  } \sum_{K \in \mathcal S_s \;:\; \ell (K)  = 2 ^{j}}  T_{K,s}b
 \Bigr\rVert_ \infty \lesssim 1.    
\end{equation*}
And, in $ L ^2 $, we have 
\begin{align*}
\Bigl\lVert  
\sum_{s \geq 0 \;:\; } \sum_{K \in \mathcal S_s : \ell(K) = 2^j } T_{K,s}b
 \Bigr\rVert_ 2 
 & \leq j  \sum_{s=0} ^{j-1} \sum_{K \in \mathcal S_s \;:\; \ell (K)  = 2 ^{j}}  \lVert T_{K,s}b\rVert_2 ^2 
 \\
 & \lesssim j \sum_{s=0} ^{j-1} \sum_{K \in \mathcal S_s \;:\; \ell (K)  = 2 ^{j}} \lvert  K\rvert ^{-1-\epsilon_d/n} \lVert b_{k-s} \mathbf 1_{K}\rVert_1 ^2 
 \\ 
 & \lesssim j 2 ^{- \epsilon_d j}  \sum_{s}  \lVert b_s\rVert_1 \lesssim 2 ^{- \frac{j\epsilon_d}{2}} \lvert  I_0\rvert 
\end{align*}
where we have used \eqref{e:standard}.  Interpolating these two estimates gives us 
\begin{equation*}
\Bigl\lVert  
\sum_{s \geq 0 \;:\; } \sum_{K \in \mathcal S_s \;:\; \ell (K)  = 2 ^{j}} T_{K,s}b
 \Bigr\rVert_ q \lesssim 2 ^{- \epsilon_d j /2q} \lvert I_0\rvert ^{1/q}.  
\end{equation*}
Summing this over $ t_n \leq j \leq s_0$ completes the proof.  
\end{proof}

It therefore remains to consider the complementary `non-standard' collection $ \mathcal N_s$. Note that for it, we have 
\begin{equation}\label{e:non}
\lVert T_{K,s}b \rVert_2 ^2 < \frac {2 C_0 } {\lvert  I\rvert }\int _{Z_I} b_s (x) b_s (y)\; dx \, dy .  
\end{equation}
There is an elementary endpoint estimate.  

\begin{lemma}\label{l:nonKnfty}
Under the assumption that 
\[\sup_{K \subset I_0} \langle g \rangle_K \leq A \langle g \rangle_{I_0},\]
we have these inequalities  below, uniformly in $ s\geq 0$
\begin{equation}\label{e:nonInfty}
\Bigl\langle 
\sup _{\epsilon } \Bigl\lvert  
\sum_{K \in \mathcal N_s \;:\; \ell (K) \geq \epsilon } T _{K,s}b \Bigr\rvert  , g
\Bigr\rangle
\lesssim   \lvert  I_0\rvert 
\langle f \rangle_{I_0}   \langle g \rangle_{I_0} . 
\end{equation}
\end{lemma}

\begin{proof}
We argue by duality. For measurable $ \sigma \;:\; I_0 \mapsto (0, \infty )$, we set 
\begin{equation*}
\tilde T_K f = \mathbf 1_{ \ell (K) \geq \sigma (x)} T_K  f , 
\end{equation*}
so that for arbitrary choice of $ \sigma $, we can estimate 
\begin{align*}
\Bigl\langle 
\sum_{K \in \mathcal N_s  } \tilde T_K b_{k-s}    , g
\Bigr\rangle
& = \sum_{K \in \mathcal N_s  }  \langle b_{k-s},  \tilde T_K ^{\ast} g\rangle 
 \lesssim  \sum_{K \in \mathcal N_s  }  \lVert b_{k-s} \mathbf 1_K\rVert_1 \langle g \rangle_K \lesssim \lvert  I_0\rvert \langle f \rangle_{I_0} \langle g \rangle_{I_0}.   
\end{align*}
\end{proof}

The essence of the argument is therefore the $ L ^2 $ bound below. 
\begin{lemma}\label{l:2}
We have these inequalities  below, uniformly in $ s\geq 0$, for a choice of $ \eta = \eta(n,d)>0$. 
\begin{equation}\label{e:2}
\Bigl\lVert  \sup _{\epsilon } \Bigl\lvert  
\sum_{K \in \mathcal N_s \;:\; \ell (K) \geq \epsilon } T_{K,s}b 
\Bigr\rvert \,\Bigr\rVert_2 \lesssim  2 ^{- \eta s} \langle f \rangle_{I_0} \lvert  I_0\rvert ^{1/2}.
\end{equation}
\end{lemma}

Interpolating between \eqref{e:nonInfty} and \eqref{e:2}, we have 
\begin{equation*}
\Bigl\langle 
\sup _{\epsilon } \Bigl\lvert  
\sum_{K \in \mathcal N_s \;:\; \ell (K) \geq \epsilon } T _{K,s}b \Bigr\rvert , g
\Bigr\rangle
\lesssim  2 ^{-s\eta/q}\lvert _0 \rvert  ^{1/q}
\langle f \rangle_{I_0}   \langle g \rangle_{I_0,q'} , \qquad s\geq 0. 
\end{equation*}
Summing this over $ s\geq 0$ completes the proof, with a single power of $ q \simeq  \frac 1 {r-1}$ as the leading coefficient in \eqref{e:2ndMain}.

\begin{proof}
There is an important calculation, relating the $L^2$ norm of $ T _{K,s}b$ to $s$.  The argument below is illustrated in Figure \ref{f:cover}.  
\begin{align} 
 \lVert T _{K,s}b\rVert_2 ^2  &
  \lesssim   \lvert  K\rvert ^{-1} \int _{Z _{K }} b_{k-s} (x) b _{k-s} (y) \; dx \, dy 
&(\ell (K)=2^k) 
 \\
 &\lesssim 
 \lvert  K\rvert ^{-1}  
\sum_{\substack{I, J \in \mathcal B _{k-s}\\ I , J \subset K }}   
\lvert I \times J\rvert \mathbf 1_{ I \times J\cap Z_K \neq \emptyset }  
& \Bigl(b _{k-s} = \sum_{I\in \mathcal B _{k-s} } b \mathbf 1_{I} \Bigr)
\\  
\noalign{\noindent here, we have used \eqref{e:st} twice.  Crucially, we  have \eqref{e:Z}, and so we can continue } 
& \lesssim   \lvert  K\rvert ^{-1}   |Z_K  + \{ (x,y)  \;:\;   \lvert  (x,y)\rvert  \leq C2^{k-s} \}|  
 \\\label{e:beta0} 
 & \lesssim   (2 ^{- \epsilon _d  k} + 2 ^{-s}) \lvert  K\rvert  .  
 \end{align}
 At the end of the arugment, we will again need this argument, but done fiberwise, using the full strength of \eqref{e:Z}. 

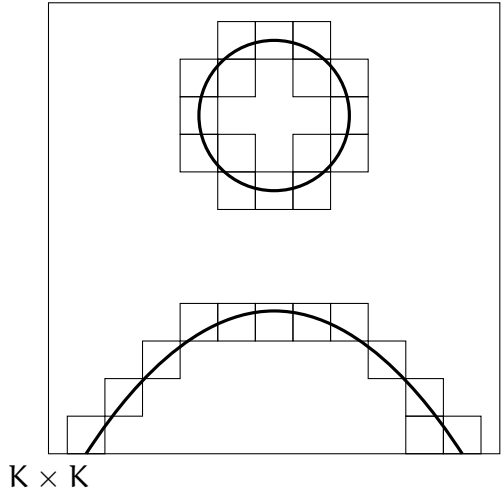
\begin{figure}
\begin{tikzpicture}
\draw[very thick]  (0, 1.5) circle (1cm); 
\draw[very thick]  (-2.5,-3) parabola bend (0,-1.1) (2.5,-3); 
\draw  (3,3) rectangle  (-3,-3) node [right,below] { $ K \times K$};  
 \foreach \a/\b/\c/\d in {-2.75/-3/ -2.25/ -2.5,  -2.25/-2.5/-1.75/-2 , -1.75/-2/-1.25/-1.5,  -1.25/-1.5/-.75/-1, -.75/-1.5/-.25/-1,
 -.25/-1.5/.25/-1, .25/-1.5/.75/-1,.75/-1.5/1.25/-1, 1.25/-1.5/1.75/-2 ,1.75/-2/2.25/-2.5 ,1.75/-2.5/2.25/-3, 2.25/-2.5/2.75/-3   ,
 -.25/2.25/.25/2.75,   -.25/.25/.25/.75,   .25/2.25/.75/2.75,   .25/.25/.75/.75,  -.25/2.25/-.75/2.75,   -.25/.25/-.75/.75,  
 -.25/2.25/-.75/1.75,   -.25/.75/-.75/1.25,   .25/2.25/.75/1.75,   .25/.75/.75/1.25, 
  -.75/2.25/-1.25/1.75,   -.75/.75/-1.25/1.25,   .75/2.25/1.25/1.75,   .75/.75/1.25/1.25,  .75/1.25/1.25/1.75 , -.75/1.25/-1.25/1.75 
 }   \draw   (\a,\b) rectangle  (\c,\d);  
\end{tikzpicture}
\caption{
The large square is $ K \times K$, with the set $  \{ (x,y) \;:\; x-y \in Z_K\}$ indicated by the two thick curves above. The set $ Z_K$ is then covered by rectangles of the form $ I \times J$ for $  I,J\in \mathcal B _{k,s}$, and the function $ b \times b$ has integral $ 2 ^{-2t }  \lvert  I \rvert \cdot \lvert  J\rvert  $ on each rectangle $ I \times J$.   
}
\label{f:cover}
\end{figure}

Make this secondary division of $ \mathcal N_s$. For integers $ t\geq 0$ set $ K\in \mathcal N _{s,t}$ if 
\begin{equation} \label{e:st}
 K_0 2 ^{-t} \leq \langle b _{k-s} \rangle_{K} < K_0 2 ^{-t+1}, \qquad  \ell (K) = 2 ^{k}. 
\end{equation}

The key additional property that we have is this Carleson measure estimate: 
\begin{equation}\label{e:CM}
\sum_{K \in \mathcal N _{s,t} \;:\; K\subset J} \lvert  K\rvert \lesssim 2 ^{t} \lvert  J\rvert  
\end{equation}
for all dyadic $ K\subset I_0$ with $K\in \mathcal{N}_{s,t}$.  Indeed, we necessarily have $ \sum_{t\geq 0 \,:\, 2^t\leq \lvert J\rvert } \langle b_t \rangle_J \lesssim 1$, so that 
\begin{align*}
\sum_{K \in \mathcal N _{s,t} \;:\; K\subset J} \lvert  K\rvert
& \lesssim 2 ^{t} 
 \sum_{t \geq 0 \,:\, 2^t\leq \lvert J\rvert }  \int _{J} b_t \lesssim 2 ^{t} \lvert  J\rvert.   
\end{align*}

The first order of business is to discard the scales and locations where the bad functions have ``large'' averages, relative to the parameter $s$:
we show that for each $t \leq \eta s$,
\begin{equation}\label{e:30} \Bigl\| \sup_\epsilon \Bigl| \sum_{ \ell(K) \geq \epsilon, K \in \mathcal{N}_{s,t}} T_{K,s} b \Bigr| \Bigr\|_2 \lesssim 2^{-\eta's }|I_0|^{1/2}.
\end{equation}

The Carleson measure condition \eqref{e:CM} implies that the set 
\begin{equation} 
E = 
\Bigl\{ \sum_{I\in \mathcal N _{s,t}} \mathbf 1_{I}  > C 2 ^{t}\Bigr\} 
\end{equation}
has measure at most $ \frac 14 \lvert  I_0\rvert $, for appropriate constant $ C$.  
This permits a further modification of \eqref{e:30}, namely we restrict the sum to 
$ \mathcal N _{s,t} ^{\sharp} := \{I \in \mathcal N _{s,t} \;:\; I\not\subset E\}$, and show 
\begin{equation}\label{e:40}
\Bigl\| \sup_\epsilon \Bigl| \sum_{ \ell(K) \geq \epsilon, K \in \mathcal{N}^{\sharp}_{s,t}} T_{K,s} b \Bigr| \Bigr\|_2 \lesssim 2^{-\eta's }|I_0|^{1/2}.
\end{equation}
This proves \eqref{e:30} upon a straight forward recursion inside the set $ E$.

But now, taking into account that $t \leq \eta s$, \eqref{e:40} now follows from an application of Cauchy-Schwarz, and the observation that for each $K \in \mathcal{N}_s$, $\|T_{K,s} b\|_2^2 \lesssim 2^{-\epsilon_d s} |K|$, since each such $K$ has $\ell(K) \geq 2^s$. The upshot is that we can proceed under the assumptions $t \geq  \eta  s$.

\smallskip

We now claim a strengthened version of \eqref{e:2}, namely 
\begin{equation}\label{e:3}
\Bigl\lVert  \sup _{\epsilon } \Bigl\lvert  
\sum_{K \in \mathcal N_{s,t}^{\sharp}  \;:\; \ell (I) \geq \epsilon } T_{K,s} b 
\Bigr\rvert \,\Bigr\rVert_2 \lesssim  2 ^{-t/3} \langle f \rangle_{I_0} \lvert  I_0\rvert ^{1/2}, \qquad s,t\geq 0. 
\end{equation}
which leads to the same conclusion for the full set $\mathcal{N}_{s,t}$ as above.

\smallskip 

The essence of this reduction is that it places  the Rademacher-Menshov inequality \eqref{e:RM} at our disposal.  Namely, after proving an appropriate orthogonality condition, an instance of \eqref{e:bessel}, we can conclude a result for maximal truncations from \eqref{e:RM}. 
Let $ \mathcal M _1$ be the minimal elements of $ \mathcal N ^{\sharp} := \bigcup _{s=0} ^{s_0} \mathcal N _{s,t} ^{\sharp}$, 
and inductively set $ \mathcal M _{u+1} $ to be the maximal elements of $ \mathcal N ^{\sharp} \setminus \bigcup _{v=1} ^{u} \mathcal M_v $.  Note that this set is empty for $ u+1 \geq u_0 =C 2 ^{t}$.  
Then, set 
$
\beta _u := \sum_{K\in \mathcal M_u} T  _{K,s}b  
$.
The required orthogonality statement is this. For any choice of constants $ \varepsilon _t \in \{-1,0,1\}$, there holds 
\begin{equation}\label{e:beta}
\Bigl\lVert \sum_{u=1} ^{u_0} \varepsilon _u \beta _u  \Bigr\rVert_2 
\lesssim  2 ^{- t/3}  \lvert  I_0\rvert ^{1/2} , \qquad  t\geq \eta s.    
\end{equation}
This is the hypothesis \eqref{e:bessel} of the Rademacher-Menshov lemma, so we conclude that 
\begin{equation*}
\Bigl\lVert \sup _{v}  \Bigl\lvert 
\sum_{u=1} ^{v} \varepsilon _u\beta _u 
\Bigr\rvert\Bigr\rVert_2 
\lesssim   t  2 ^{-t/3} \lvert  I_0\rvert ^{1/2}. 
\end{equation*}
And this implies \eqref{e:3}.

First, observe that   
\begin{align} 
\sum_{u=1}^{C2^t} 
\lVert \beta _u \rVert_2 ^2 
& = \sum_{K\in  \mathcal N_{s,t}}  \lVert T _{K,s}b\rVert_2 ^2  
 \lesssim  \sum_{K\in  \mathcal N_{s,t}} \langle b_{k-s}\rangle_K^2   \lvert  K\rvert 
\\  \label{e:beta3}
& \lesssim  2 ^{-t}  \sum_{K\in  \mathcal N_{s,t}} \int_K  b_{k-s} \;dx \lesssim 2^{-t} \lvert I_0\rvert. 
\end{align}

It remains to consider $ u < v$, and the inner product 
\begin{align} \label{e:uv}
\langle \beta _u , \beta _v  \rangle 
& = \sum_{J\in \mathcal M_u}   \sum_{\substack{K\in \mathcal M_v \\ J\subset K }} 
\langle  b _{j -s}  , T _J ^{\ast} T _{K} b _{k-s} \rangle  & (\ell J = 2 ^{j},\ \ell K = 2 ^{k}) . 
\end{align}
 The kernel of $T _J ^{\ast} T _{K}  $ is controlled by \eqref{e:same}. There are two terms on the right in \eqref{e:same}, 
and for the second we have 
\begin{align} 
\sum_{J\in \mathcal M_u}   \sum_{\substack{K\in \mathcal M_v \\ J\subset K }} 
\lvert  K\rvert ^{-1- \epsilon _d/n}   \int _{K} b _{k-s}  \cdot \int _{J} b _{j-s} 
& \lesssim 
2 ^{-2t}\sum_{J\in \mathcal M_u}   \sum_{\substack{K\in \mathcal M_v \\ J\subset K }} 
\lvert  K\rvert ^{-1- \epsilon _d/n}    \lvert  K \rvert \cdot \lvert  J\rvert
\\\label{e:uv1}
& \lesssim 2 ^{-2t- \epsilon _d  v} \lvert  I_0\rvert.     
\end{align}
The depends upon \eqref{e:st}, and the fact that $ K \in \mathcal M_v $ implies $ \ell (K) \geq 2 ^{v}$.

The first term on the right in \eqref{e:same} is the essential oscillatory term.  
It involves the set $ Z_K$, and our initial estimate is as below, fixing the interval $K$. 
\begin{align} \label{e:ZT}
 \sum_{\substack{J\in \mathcal M_u\\ J\subset K}} 
\langle  b _{j -s}  , T _J ^{\ast} T _{K} b _{k-s} \rangle  
 &\lesssim 
 \lvert  K\rvert ^{-1}  \sum_{\substack{J\in \mathcal M_u\\ J\subset K}} 
\int_K   b _{k-s} (x) \int _{J} b _{j-s}(y) 
 \mathbf 1_{Z _{K} } (x,y) \;dy \, dx 
\end{align}  
Fix $x$ above, and recall that $\pi_x Z$ is the fiber of $Z$ over $x$. 
The integral in $y$ is over the set $\pi_x Z_K $, and the integral is 
\begin{align}
\sum_{\substack{J\in \mathcal M_u\\ J\subset K}} 
 \int _{J \cap \pi_x Z_K } b _{j-s}(y) \;dy & \lesssim 
 2^{-t} \vert \pi_x Z_k + \{ |y| \lesssim 2^{k - v+u} \} \vert
 \\ \label{e:ZTT}
 &\lesssim 2^{-t}\lvert K\rvert (2^{- \epsilon_s v} + 2 ^{-v+u}). 
\end{align}
Above, we appeal to the fact that $\lvert J\rvert \leq 2 ^{-v+u}\lvert K \rvert$, and \eqref{e:st}, and the condition \eqref{e:Neighborhood}, which is an estimate uniform over all fibers.  
It follows that we have 
 \begin{align}
 \eqref{e:ZT} 
 & \lesssim 2 ^{-2t }  (2 ^{- \epsilon _d v}  + 2^{-s- \lvert  u-v\rvert }) 
 \lvert  K\rvert.  
\end{align}
Sum this estimate  over $K\in \mathcal{M}_v$, to conclude the bound 
\begin{equation}
\lvert \langle \beta _u , \beta _v  \rangle \rvert 
\lesssim 
2^{-2t - \epsilon_d \lvert u-v\rvert} \lvert I_0\rvert . 
\end{equation}
Combine this with  \eqref{e:beta3} easily prove \eqref{e:beta}, completing the proof of our Lemma.

\end{proof}

\begin{lemma}\label{l:simple} For any collection $ \mathcal I \subset \mathcal D_+$ we have 
\begin{equation}\label{e:simple}
 \lVert T _{\ast , \mathcal I}  \;:\; L ^{q} \to L ^{q        }\rVert \lesssim q , \qquad 2\leq q < \infty . 
\end{equation}
\end{lemma}
 
\begin{proof}
Observe that for fixed scales, we have 
\begin{align*}
\lVert T _{\mathcal I (k)} \;:\; L ^{\infty } \mapsto L ^{\infty }\rVert &\lesssim 1 , 
\\
\lVert T _{\mathcal I (k)} \;:\; L ^{2 } \mapsto L ^{2 }\rVert &\lesssim  2 ^{- \eta k}, \qquad k\geq 0. 
\end{align*}
The first estimate is trivial, and the second is a consequence of the oscillatory estimate \eqref{e:same}.  
Interpolating these estimates, and adding up gives the proof. 
\end{proof}

\bibliographystyle{alpha,amsplain}

\end{document}